\documentclass[12pt]{amsart}

\usepackage[utf8]{inputenc}
\usepackage{fullpage, mathrsfs, mathtools, amssymb, tikz,bm, hyperref, amscd, scrextend, tikz-cd, verbatim, enumerate}
\usepackage[inline]{enumitem}
\usepackage{microtype}
\usetikzlibrary{matrix, arrows}

\input xy
\xyoption{all}

\newtheorem{question}{Question}
\newtheorem{theorem}{Theorem}[section]

\newtheorem{prop}[theorem]{Proposition}
\newtheorem{cor}[theorem]{Corollary}

\theoremstyle{definition}
\newtheorem{definition}[theorem]{Definition}
\newtheorem{example}[theorem]{Example}
\newtheorem{remark}[theorem]{Remark}
\newtheorem{conj}{Conjecture} 

\theoremstyle{remark}

\newcommand{\Q}{{\mathbb Q}}

\newcommand{\Z}{{\mathbb Z}}

\newcommand{\PP}{{\mathbb P}}

\newcommand{\calK}{\mathcal K}
\newcommand{\calO}{\mathcal O}
\newcommand{\calM}{\mathcal M}
\newcommand{\calF}{\mathcal F}

\newcommand{\calA}{\mathcal A}

\newcommand{\calE}{\mathcal E}
\newcommand{\calU}{\mathcal U}
\newcommand{\calX}{\mathcal X}
\newcommand{\calC}{\mathcal C}
\newcommand{\calS}{\mathcal S}

\newcommand{\Spec}{\mathrm{Spec}}
\newcommand{\Proj}{\mathrm{Proj}}

\newcommand{\mc}[1]{\mathcal{#1}}
\newcommand{\mb}[1]{\mathbb{#1}}

\title{Moduli of fibered surface pairs from twisted stable maps}
\author[Ascher]{Kenneth Ascher}
\email{kenneth\_ascher@brown.edu}
\author[Bejleri]{Dori Bejleri}
\email{dbejleri@math.brown.edu}

\begin{document}

\maketitle
\begin{abstract} In this paper, we use the theory of twisted stable maps to construct compactifications of the moduli space of pairs $(X \to C, S + F)$ where $X \to C$ is a fibered surface, $S$ is a sum of sections, $F$ is a sum of marked fibers, and $(X,S+F)$ is a stable pair in the sense of the minimal model program. This generalizes the work of Abramovich-Vistoli, who compactified the moduli space of fibered surfaces with no marked fibers.  Furthermore, we compare our compactification to Alexeev's space of stable maps and the KSBA compactification.  As an application, we describe the boundary of a compactification of the moduli space of elliptic surfaces.    \end{abstract}

\section{Introduction}

A fibered surface is a flat proper morphism $f: X \to C$ from a smooth projective surface to a smooth projective curve with sections $\sigma_1, \ldots, \sigma_\nu$, and generic fiber a stable $\nu$-pointed curve of genus $\gamma$. As surfaces fibered in curves appear naturally (e.g. elliptic surfaces), it is natural to ask for geometric compactifications of their moduli. Abramovich and Vistoli \cite{av} used the theory of twisted stable maps to Deligne-Mumford stacks to construct a compactification $\calK_{g}(\overline{\calM}_{\gamma,\nu})$ of the moduli space of fibered surfaces.

In this paper we use the general theory of twisted stable maps developed in \cite{av2} to extend their results in \cite{av} to the pairs case, and construct compactifications of the moduli space of pairs $(f: X \to C, S + F)$ where $(X \to C, \sigma_1, \ldots, \sigma_\nu)$ is a fibered surface,
$$
S + F = \sum_{i =1}^\nu S_i + \sum_{j = 1}^n F_j
$$
is a sum of sections $S_i := \sigma_i(C)$ and marked fibers $F_j$ with their reduced structure, and $(X, S + F)$ is a stable pair in the sense of the minimal model program (see Definition \ref{def:stablepair}). Using this approach, we are able to describe the stable reduction process (see Theorem \ref{thm:introelliptic}) in a more straightforward manner than solely using techniques from the minimal model program (MMP).

Our first theorem generalizes Propositions 6.10 and 6.13 of \cite{av}. See also Corollary 1.8.2 of \cite{abramovich}.

\begin{theorem}[See Theorem \ref{thm:mainthm} and Corollary \ref{cor:morphism}] \label{theorem1}There exists a morphism 
$$
\varphi: \calK_{g,n}(\overline{\calM}_{\gamma, \nu}) \to \calA_{\bf{v}}(\overline{M}_{\gamma,\nu})
$$
from the space of twisted stable maps with target $\overline{\calM}_{\gamma, \nu}$ to the space of Alexeev stable maps with target  $\overline{M}_{\gamma,\nu}$.  

\end{theorem} 

The space $\calA_{\bf{v}}(V)$ of Alexeev stable maps to a projective scheme $V$ is a higher dimensional analogue of the moduli space of Kontsevich stable maps $\overline{\calM}_{g,n}(V)$. It parametrizes maps $g : (X,D) \to V$ from a semi-log canonical (slc) surface pair $(X,D)$ such that $\omega_X(D)$ is a $g$-ample $\Q$-line bundle with numerical data $\bf{v}$ as constructed in \cite{mgnw} (see Theorem \ref{alexeev} for more details). We fix the degree of the twisted stable map throughout, which we suppress from the notation;
$\bf{v}$ depends on this choice of degree. Finally, we note that the fact that a twisted stable map gives rise to an Alexeev stable map was essentially proven in Corollary 1.8.2 of \cite{abramovich}. \\

We identify the surface pairs in the image of the morphism $\varphi$ and call these \emph{twisted surfaces} (see Definition \ref{def:twistedsurface}). As a consequence of properness for the moduli space of twisted stable maps, we obtain the following. 

\begin{theorem}[See Theorem \ref{thm:thm2}]\label{thm:2} There exists a proper Deligne-Mumford stack $\calF^{\bf{v}}_{g,n}(\overline{M}_{\gamma,\nu})$ with projective coarse moduli space parameterizing pairs $(X \to C \to \overline{M}_{\gamma,\nu}, S + F)$ where:
\begin{enumerate} \item $(X \to C, S + F)$ is a twisted fibered surface, and 
\item $(X, S + F) \to \overline{M}_{\gamma,\nu}$ is an Alexeev stable map with numerical data $\bf{v}$. \end{enumerate} The map $\varphi$ factors as $\calK_{g,n}(\overline{\calM}_{\gamma,\nu}) \to \calF^{\bf{v}}_{g,n}(\overline{M}_{\gamma,\nu}) \to \calA_{\bf{v}}(\overline{M}_{\gamma,\nu})$ where the first map is surjective and the second forgets the fibration $X \to C$. \end{theorem} 

It is an interesting question to ask if the morphism $\calF^{\bf{v}}_{g,n}(\overline{M}_{\gamma,\nu}) \to \calA_{\bf{v}}(\overline{M}_{\gamma,\nu})$ is surjective onto the irreducible components of $\calA_{\bf{v}}(\overline{M}_{\gamma,\nu})$ that it hits. This amounts to the following deformation-theoretic question. Given a \emph{normal} twisted fibered surface $(f : X \to C, S + F)$, does a small deformation of $X$ extend to a small deformation of $f : X \to C$? This question is answered in the affirmative in arbitrary dimension but without marked fibers in \cite{fsv}. \\

Now we briefly explain the idea behind using twisted stable maps to understand fibered surfaces pairs. Since the generic fiber of a fibered surface $f: X \to C$  with $\nu$ sections is a $\nu$-pointed genus $\gamma$ stable curve, there is a naturally induced rational map $C \dashrightarrow \overline{\calM}_{\gamma, \nu}$, to the moduli stack of stable curves. That is, there exists a nonempty open subset $U \subset C$ such that $X|_U \to U$ is a family of stable curves inducing a morphism $U \to \overline{\calM}_{\gamma,\nu}$. 

Since $\overline{\calM}_{\gamma,\nu}$ is a proper Deligne-Mumford stack, the rational map $C \dashrightarrow \overline{\calM}_{\gamma, \nu}$ extends to a morphism $\calC \to \overline{\calM}_{\gamma,\nu}$ from a smooth projective orbifold curve $\calC$ with marked points $\Sigma^\calC_i \subset \calC \setminus U$, where $\calC$ has coarse space $C$ containing $U$ as an open dense subset. 

We now have a fiber product diagram, where $\calU \to \overline{\calM}_{\gamma,\nu}$ represents the universal family:
$$
\xymatrix{\calX \ar[r] \ar[d] & \calU \ar[d] \\ \calC \ar[r]  & \overline{\calM}_{\gamma,\nu}}
$$
Then $f': X' \to C$, the coarse space of $\calX \to \calC$, will be a fibered surface pair birational to $f: X \to C$, with sections $S_i'$ and marked fibers $F_i$ above the marked points $\Sigma_i^C$. 

From this construction, the map $\calC \to  \overline{\calM}_{\gamma,\nu}$ induces  a birational model $(X' \to C, S + F)$ of a fibered surface pair, along with a morphism $C \to \overline{M}_{\gamma, \nu}$. Now using twisted stable maps, we compactify the space of such maps by allowing the source curve $\calC$ to degenerate to a nodal orbifold curve and obtain a morphism to $\calA(\overline{M}_{\gamma,\nu})$ by taking coarse space of the pullback of the universal family of $\overline{\calM}_{\gamma,\nu}$.
\subsection{Stable maps versus stable pairs} 

One is often interested in the space of stable surface \emph{pairs} instead of stable maps. For a \emph{twisted} fibered surface $(X \to C, S + F)$ (see Definition \ref{def:twistedsurface}) over a smooth curve $C$ such that $(X, S + F)$ is a stable pair, there is a uniquely determined map $C \to \overline{M}_{\gamma, \nu}$ by taking the coarse space of $\calC \to \overline{\calM}_{\gamma, \nu}$ as above. In this case, no information is lost by forgetting the map. However, this may no longer be true on the boundary of the moduli space since the source of a stable map may not be stable as a pair, i.e. the pair $(X, S + F)$ might not be stable. 

We let $\calF^v_{g,n}(\gamma, \nu)$ denote the compactification of the space of twisted stable fibered surface pairs $(X \to C, S + F)$ provided by the minimal model program (see Definition \ref{def:ffunctor}). 

\begin{theorem}[See Theorem \ref{prop:thm3} and Corollary \ref{cor:proper}] The moduli space $\calF^v_{g,n}(\gamma, \nu)$ is a proper Deligne-Mumford stack equipped with a morphism $\calF^v_{g,n}(\gamma,\nu) \to \overline{\calM}_{g,n}$. \end{theorem} 

More precisely, we lift a one parameter family of twisted stable fibered surfaces in the interior of $\calF^v_{g,n}(\gamma,\nu)$ to a family of twisted stable maps. There is then a unique limit in $\calF^v_{g,n}(\gamma,\nu)$ by running stable reduction on the family of twisted stable maps, taking the coarse space, and then running the minimal model program. 

An interesting question is to determine which surfaces appear on the boundary of $\calF^v_{g,n}(\gamma,\nu)$. That is, what are the stable limits of families of $\nu$-pointed genus $\gamma$ fibrations over an $n$-pointed genus $g$ stable curve provided via MMP? 

Using the work of La Nave \cite{ln}, we can answer this question in the case when $\gamma = \nu = 1$, i.e. the case of elliptic surfaces. The starting point is the analysis of log canonical models of elliptic surfaces with a section and marked fibers carried out in \cite{calculations}. Using these results, we describe the twisted elliptic surface pairs appearing on the boundary explicitly. Furthermore, La Nave in \cite{ln} describes how the steps of the minimal model program affect the central fiber of the stable limit of the corresponding twisted stable maps. As a result, we obtain a complete description of the boundary components of $\calF^v_{g,n}(1,1)$: 

\begin{theorem}[see Theorem \ref{thm:elliptic}]\label{thm:introelliptic} The boundary of $\calF^v_{g,n}(1,1)$ parametrizes broken elliptic surfaces (see Definition \ref{def:brokenelliptic}) consisting of an slc union of elliptic components and trees of pseudoelliptics (see Definition \ref{def:pseudoelliptic}) glued along fibers. \end{theorem}

The above theorem is key to understanding moduli spaces of $\calA$-weighted stable elliptic surfaces which we study in \cite{kennydori1}. We note that Brunyate \cite{brunyate} used similar methods to describe the boundary of the moduli space of elliptic K3 surfaces in the case where the section and singular fibers are marked with very small coefficients. \\ 

In passing from $\calF^{\bf{v}}_{g,n}(\overline{M}_{\gamma, \nu})$ to $\calF^v_{g,n}(\gamma, \nu)$, one begins with the data $(f : X \to C \to \overline{M}_{\gamma,\nu}, S + F)$ of a twisted fibered surface pair and an Alexeev stable map to $\overline{M}_{\gamma,\nu}$, then runs the minimal model program on the pair $(f : X \to C, S + F)$ to make $\omega_X(S + F)$ ample in the absolute sense rather than ample relative to the morphism $X \to \overline{M}_{\gamma,\nu}$. It is natural to ask if this process is functorial:

\begin{question} Is there a morphism $\calF^{\bf{v}}_{g,n}(\overline{M}_{\gamma,\nu}) \to \calF^v_{g,n}(\gamma,\nu)$ that forgets the map to $\overline{M}_{\gamma,\nu}$? \end{question} 

The difficulty in answering the above question is that the steps of the minimal model program are not always compatible with arbitrary base change. Thus one needs a strong vanishing theorem or a fine study of the actual birational transformations that occur in families of twisted fibered surfaces when stabilizing. In \cite{kennydori1}, we carry this out in the case of elliptic surfaces $\gamma = \nu = 1$. 

\subsection{Enumerative geometry} 

Using the theory of twisted stable maps, Abramovich-Graber-Vistoli construct Gromov-Witten invariants of Deligne-Mumford stacks in the algebraic setting. In our case, these are computed by virtual invariants on $\calK_{g,n}(\overline{\calM}_{\gamma, \nu})$. The map in Theorem \ref{theorem1} suggests that one may be able to use birational geometry of fibered surfaces to understand Gromov-Witten invariants of $\overline{\calM}_{\gamma, \nu}$. \\

This leads us to ask the following question, and make the following conjecture:

\begin{question} Do the spaces  $\calF^v_{g,n}(\gamma,\nu)$ admit a virtual fundamental class? \end{question}

\begin{conj} The map $\calK_{g,n}(\overline{\calM}_{\gamma,\nu}) \to \calF^{\bf{v}}_{g,n}(\overline{M}_{\gamma,\nu})$ in Theorem \ref{thm:2} is a \emph{virtual normalization}. That is, it induces an equivalence between natural virtual fundamental classes. \end{conj} 

One might also expect to construct weighted invariants of $\overline{\calM}_{\gamma, \nu}$ as in \cite{alexeevguy}. However, weighted maps to stacks are not well behaved and it is unclear how to compactify the moduli space. The minimal model program suggests a way to get around this. 

Indeed, in \cite{kennydori1} we study moduli spaces $\calE_{v,\calA}$ of $\calA$-weighted stable elliptic surfaces. These are replacements of the moduli space of weighted stable maps to $\overline{\calM}_{1,1}$. One can hope to construct an analogue of \emph{quasi-map} invariants of $\overline{\calM}_{1,1}$ using these spaces.

\begin{question} Do the spaces $\calE_{v, \calA}$ admit a virtual fundamental class? \end{question}

The difficulty in answering the above question is the fact that for arbitrary coefficients $\calA$, the infinitesimal structure of the moduli space of stable pairs has not been developed. Indeed, we do not have a good deformation theory of stable pairs with arbitrary coefficients. We hope the considerations in this paper and \cite{kennydori1} will help better understand this phenomenon. 

\subsection{Higher dimensions} In this paper we restrict to fibered surfaces in order to avoid the subtleties involved with slc singularities and moduli of stable pairs in higher dimensions. There should not, however, be serious difficulties in extending the results of this paper to higher dimensions assuming we have access to a suitable moduli space of stable pairs. 

More precisely, given a Deligne-Mumford stack $\calM_v$ of $n$-dimensional stable pairs $(X,D)$ with volume $v$ and projective coarse moduli space $M_v$, one can use twisted stable maps to $\calM_v$ to construct a moduli space of $(n + 1)$-dimensional pairs $(X \to C \to M_v, D + F)$. Here $(X \to C, D)$ a flat family with generic fiber of type $\calM_v$ with coarse map $C \to M_v$, the divisor $F$ is a sum of reduced marked fibers, and $(X, D + F) \to M_v$ is an Alexeev stable map. 

This is carried out inductively starting with curves in  \cite{abramovich} to study moduli spaces and stable reduction of higher dimensional pairs equipped with a plurifibration.

\subsection{Determining the main components} 

An important question that arises when compactifying moduli spaces is to characterize the main components moduli theoretically. In our case, one may be interested in determining the closure in $\calF^{\bf{v}}_{g,n}(\overline{M}_{\gamma,\nu})$ of the locus of normal fibered surfaces. An often fruitful approach is to use logarithmic geometry to characterize the main component as a moduli of objects endowed with a natural log structure (see \cite{fkato} for the case of stable curves and \cite{loggeom} for an introduction to logarithmic geometry). 

This approach almost works in our setting. Abramovich-Vistoli characterize the smoothable nodal orbifold curves as those that are \emph{balanced}, that is, \'etale locally isomorphic to 
$$
\left[\mathrm{Spec} (k[x,y]/(xy))/\boldsymbol \mu_r \right]
$$
around each node where $\boldsymbol \mu_r$ is the group of $r$-th roots of unity acting by $(x,y) \mapsto (\alpha x, \alpha^{-1} y)$. It is shown in \cite{olsson} that a family of balanced twisted curves can be endowed with a canonical log structure making it a log smooth morphism. It follows that the fibered surface stack $\calX \to \calC$ induced by a \emph{balanced} twisted stable map $\calC \to \overline{\calM}_{\gamma,\nu}$ carries a canonical log structure making $\calX$ a log smooth Deligne-Mumford stack.

However, it is not necessarily true that the coarse space $X$ of $\calX$ is a log smooth surface. Therefore the best we can do is describe the main component of $\calF^{\bf{v}}_{g,n}(\overline{M}_{\gamma,\nu})$ as parametrizing coarse spaces of log smooth fibered surface stacks. One may hope that a better understanding of the infinitesimal structure of the morphism $\calK_{g,n}(\overline{\calM}_{\gamma,\nu}) \to \calF^{\bf{v}}_{g,n}(\overline{M}_{\gamma,\nu})$ can be used to exploit this fact in studying the boundary of the main component.

\subsection{Conventions} 

We work over a field of characteristic 0 convenience. We expect the results of this paper to hold in positive, possibly large enough, characteristic. Indeed the theory of twisted stable maps has been developed in arbitrary characteristic in \cite{aov}. The main difficulties for moduli of surfaces in positive characteristic are discussed in \cite{mgnw} and many of these have since been addressed in the literature.    

\subsection*{Acknowledgements}We thank our advisor Dan Abramovich for his constant support. We thank Valery Alexeev and Dhruv Ranganathan for helpful suggestions. Research of both authors supported in part by funds from NSF grant DMS-1500525.

\section{Moduli of stable pairs}
	
We begin with the relevant background from the theory of stable pairs, stable maps and their moduli. The starting point for compact moduli spaces of higher dimensional varieties is the higher dimensional generalization of a stable curve.  
	
	\begin{definition}\label{def:loc} Let $(X, D)$ be a pair of a normal variety and a $\Q$-divisor such that $K_X + D$ is $\Q$-Cartier. Suppose that there is a log resolution $f: Y \to X$ such that $$K_Y + \sum a_E E = f^*(K_X + D),$$ where the sum goes over all irreducible divisors on $Y$. We say that the pair $(X,D)$ has \textbf{log canonical singularities} (or is lc) if all $a_E \leq 1$. \end{definition}
	
	\begin{definition}\label{def:slc} Let $(X, D)$ be a pair of a reduced variety and a $\Q$-divisor such that $K_X + D$ is $\Q$-Cartier. The pair $(X,D)$ has \textbf{semi-log canonical singularities} (or is slc) if:
		\begin{itemize}
			\item The variety $X$ is S2,
			\item $X$ has only double normal crossings in codimension 1, and 
			\item If $\nu: X^{\nu} \to X$ is the normalization, then the pair $(X^{\nu}, \sum d_i \nu_*^{-1}(D_i) + D^{\nu})$ is log canonical, where $D^{\nu}$ denotes the preimage of the double locus on $X^{\nu}$. \end{itemize}
	\end{definition}
	
	\begin{definition} \label{def:stablepair} A pair $(X, D)$ of a projective variety and $\Q$-divisor is a \textbf{stable pair} if:
		\begin{enumerate}
			\item $(X,D)$ is an slc pair, and
			\item $\omega_X(D)$ is ample.
		\end{enumerate}
	\end{definition}

	There is also a notion of stable maps for surface pairs due to Alexeev \cite{mgnw}.
	
	\begin{definition}\label{def:stablemap} Let $X$ be a connected projective surface, let $D \subset X$ be a divisor, and let $M \subset \PP^r$ be a projective scheme. Then the morphism $f: X \to M$ is called an \textbf{Alexeev stable map} of the pair $(X,D)$ to $M$ if:
		\begin{enumerate}
			\item the pair $(X,D)$ is slc (in particular $(\omega_X(D))^{[m]}$ is invertible for some integer $m >0 $);
			\item the line bundle $\omega_X(D)^{[m]}$ is $f$-ample. 
		\end{enumerate} \end{definition}
		
		\begin{remark} Note that condition $(2)$ in Definition \ref{def:stablemap} is equivalent to the statement that the line bundle $\omega_X(D)^{[m]} \otimes f^*\calO_M(mn)$ is ample for sufficiently large $n$. This is \emph{independent} of the choice of projective embedding of $M$ since condition 2 evidently is. \end{remark} 
		
		Given a stable map $f: X \to M$ from an slc surface pair $(X,D)$, one has a well defined triple of rational numbers: 
		$$
		v_1 = c_1(\omega_X(D))^2, \hspace{3ex} v_2 = c_1(\omega_X(D))\cdot c_1(f^*\calO_M(n)), \hspace{3ex} v_3 = c_1(f^*\calO_M(n))^2.
		$$
		When $D$ is a reduced divisor, there exists a projective moduli space of stable maps: 
		
	\begin{theorem}\cite{mgnw}\label{alexeev} Given rational numbers  $v_1, v_2$ and $v_3$, there is a Deligne-Mumford stack $\calA_{\bf{v}}(M)$ admitting a projective coarse moduli space for stable maps $f: (X,D) \to M$ with invariants $\bf{v}$ $\ = (v_1,v_2,v_3)$ where $(X,D)$ is an slc surface pair and $D$ is a reduced divisor. \end{theorem}

\section{Twisted stable maps of Abramovich-Vistoli} Here we introduce the space of twisted stable maps of Abramovich and Vistoli. We urge the interested reader to consult \cite{av2}. 

Recall that if $Y \subset \PP^r$ is a projective variety with polarization $\calO_Y(1)$, Kontsevich constructed a proper Deligne-Mumford stack 
$$
\calK_{g,n}(Y,d):= \overline{\calM}_{g,n}(Y, d)
$$
with projective coarse moduli space $\bm{K}_{g,n}(Y,d)$ parametrizing degree $d$ stable maps from a genus $g$ curve into $Y$. Now replace $Y$ by a proper tame Deligne-Mumford stack $\calM$ with a projective coarse moduli space $\bm{M}$ and polarization $\calO_{\bm{M}}(1)$. In \cite{av2} the space $\calK_{g,n}(\calM,d)$ of \emph{$n$-pointed twisted stable maps $\calC \to \calM$ of degree $d$ and genus $g$} is defined as follows.

 \begin{definition} A \textbf{twisted nodal $n$-pointed curve of genus $g$} over a scheme $T$ is a diagram
$$
\xymatrix{\Sigma_i^\mathcal{C} \ar[r] \ar[rd] & \mathcal{C} \ar[d] \\ & C \ar[d] \\ & T}
$$
where
\begin{enumerate}[label = \arabic*)]
\item $\mathcal{C}$ is a tame DM stack, proper over $T$, and \'etale locally is a nodal curve over $T$;
\item $\Sigma_i^{\mathcal{C}}$ $(i = 1, \ldots, n)$ are disjoint closed substacks in the smooth locus of $\mathcal{C} \to T$;
\item $\Sigma_i^{\mathcal{C}} \to T$ are \'etale gerbes;
\item $\mathcal{C} \to C$ is the coarse space map;
\item $\mathcal{C} \to C$ is an isomorphism over the generic point of each component of $C$;
\item the coarse space $C \to T$ is a family of genus $g$ curves overs $T$.
\end{enumerate}

\end{definition}

By the tame assumption on $\calC$, it follows that $C \to T$ is a flat family connected nodal curves over $T$ (see \cite[Proposition 4.1.1]{av2}) and that the coarse space $\Sigma_i^C$ of the gerbe $\Sigma_i^\calC$ embeds into $C$. This makes $(C, \Sigma_i^C) \to T$ into an $n$-pointed nodal curve of genus $g$ over $T$. 

\begin{definition} An \textbf{$n$-pointed twisted stable map $f : (\calC, \Sigma_i^\calC) \to \calM$ of genus $g$ and degree $d$ over $T$} is a diagram
$$
\xymatrix{\calC \ar[r] \ar[d] & \calM \ar[d] \\ C \ar[r] \ar[d] & \bm{M} \\ T &}
$$
where 
\begin{enumerate}[label = \arabic*)]
\item $(\calC \to C \to T,\Sigma_i^\calC)$ is a twisted nodal $n$-pointed curve of genus $g$ over $T$;
\item $\calC \to \calM$ is a representable morphism of stacks, and
\item the coarse map $(C, \Sigma_i^C) \to \bm{M}$ is an $n$-pointed genus $g$ stable map of degree $d$ over $T$.
\end{enumerate}

\end{definition}

Then we have the following: 

\begin{theorem}\cite[Theorem 1.4.1]{av2} There is a proper Deligne-Mumford stack $\calK_{g,n}(\calM,d)$ parametrizing $n$-pointed twisted stable maps of genus $g$ and degree $d$ to $\calM$.  \end{theorem}

\begin{remark}\label{rmk:degree}
We will supress the degree $d$ in the notation of $\calK_{g,n}(\calM)$ for convenience. \end{remark}

 \section{Fibered surfaces from twisted stable maps} \label{sec:av}

In this section, we use twisted stable maps from \cite{av2} to extend the fibered surface results of \cite{av} to stable fibered surface pairs. As a result, we obtain birational models of stable fibered surface pairs $(f:X \to C, S + F)$ with section $S$ and reduced marked fibers $F$, and compare these to the log canonical models provided by the minimal model program.

\subsection{From twisted stable maps to fibered surfaces} In \cite{av}, a complete moduli of fibered surfaces is constructed as a special case of the moduli space of twisted stable maps where the target stack is taken to be 
$
\calM = \overline{\calM}_{\gamma,\nu},
$
the stack of $\nu$-pointed genus $\gamma$ curves. Indeed, a twisted stable map $\calC \to \overline{\calM}_{\gamma,\nu}$ from an \emph{unmarked} twisted nodal curve $\calC$ gives rise to a fibered \emph{stack-like surface} $(\calX \to \calC, \calS_1, \ldots, \calS_\nu)$  (see Definition \ref{def:stacklikesurface}) with sections $\calS_i$ by pulling back the universal family. The coarse space $(X \to C, S_1 + \ldots + S_\nu)$ is then a fibered surface with sections $S_i$ and genus $\gamma$ fibers. 

It is proven in Proposition 6.13 of  \cite{av} that $X \to C \to \overline{M}_{\gamma,\nu}$ is an Alexeev stable map (see Definition \ref{def:stablemap}). In this way, one obtains a morphism 
to the stack of stable maps from a surface fibered in genus $\gamma$ curves with $\nu$ sections $(X \to C, S_1 + \ldots + S_\nu)$ over a genus $g$ curve $C$, to the coarse moduli space of stable curves $\overline{M}_{\gamma,\nu}$. Furthermore, this morphism is finite on each connected component of the source. \\ 

Our goal is to understand what happens when we additionally consider marked points and marked fibers. That is, we consider an slc surface fibered in genus $\gamma$ curves $$(f : X \to C, \sum_{i = 1}^\nu S_i + \sum_{j = 1}^n F_j)$$ with $\nu$ sections $S_i$ and $n$ reduced marked fibers $F_j$ over a genus $g$ curve $C$. 

Let $(\pi: \calC \to T, \Sigma_i^{\calC})$ be a twisted nodal $n$-pointed curve over $T$ and let $\varphi : \calC \to \overline{\calM}_{\gamma,\nu}$ be a representable morphism. Consider the pullback of the universal family $\calU \to \overline{\calM}_{\gamma,\nu}$:

$$
\xymatrix{\calX \ar[r] \ar[d] & \calU \ar[d] \\ \calC \ar[r] & \overline{\calM}_{\gamma,\nu}}
$$

 This gives a family of genus $\gamma$ nodal curves $f : \calX \to \calC$ over $\calC$ with sections $\calS_1, \ldots, \calS_\nu$. In particular, $\calX$ is a tame stack with trivial stabilizer at the generic point so that $X^0 \to C^0$ is a family of nodal curves. Here $X$ and $C$ are the coarse spaces and $X^0$ and $C^0$ are the necessarily non-empty open loci over which the coarse map is an isomorphism; that is, these are the \emph{non-stacky loci}. 

\begin{definition}\label{def:stacklikesurface} A tuple $(f : \calX \to \calC \to T,\calS_i, \Sigma_i^\calC)$ as above is a \textbf{stack-like surface} fibered in genus $\gamma$ curves with $\nu$ sections and marked fibers over $T$. \end{definition}

Taking the coarse space of a stack-like fibered surface gives us a flat family of fibered surfaces $X \to C \to T$ with sections $S_1, \ldots, S_{\nu}$ so that the generic fiber of $X \to C$ is a stable $\nu$-pointed genus $\gamma$ curve. Furthermore, $C$ comes with marked points $\Sigma_i^C$ and we can take 
$
F_i := f^{-1}(\Sigma_i^C)_{red}
$
to be reduced marked fibers on $X$. Then we have that
$$
(X \to C \to T, \sum_{i = 1}^\nu S_i + \sum_j^n F_j)
$$
is a flat family of fibered surface pairs over $T$ with marked sections, marked fibers, and stable generic fiber. Furthermore, there is a canonical coarse space map $C \to \overline{M}_{\gamma,\nu}$ which gives rise to a map $(X, \sum_{i = 1}^\nu S_i + \sum_j^n F_j) \to \overline{M}_{\gamma,\nu}$ by composition. 

\begin{theorem}[See also Proposition 1.8.1 of \cite{abramovich}]\label{thm:mainthm} Suppose $\psi : (\calC, \Sigma_i^\calC) \to \overline{\calM}_{\gamma,\nu}$ is a pointed twisted stable map. Then the coarse space map $(X, \sum_{i=1}^\nu S_i + \sum_{j=1}^n F_j)\to \overline{M}_{\gamma,\nu}$ of the corresponding stack-like fibered surface is a stable map in the sense of Alexeev. \end{theorem}

\begin{proof}[Proof of Theorem \ref{thm:mainthm}] We break the proof of this theorem into two parts-- first we show that the pair is slc, and then show that the pair satisfies the relative positivity assumption.

\begin{prop} The coarse space $(X \to C, \sum_{i = 1}^\nu S_i + \sum_{j = 1}^n F_j)$ of a stack-like fibered surface is an slc pair. \end{prop} 

\begin{proof} 

The condition of being slc is \'etale local, so we may replace $X$ by an \'etale neighborhood which has a global chart of $\calX \to \calC$. That is, we have a family of stable curves $Y \to U$ with sections $S_i'$ over a marked nodal curve $(U,\Sigma_j^U)$ as well as marked fibers $G_j$ above $\Sigma_j^U$. Furthermore, 
$$
(Y, \sum_{i = 1}^\nu S_i' + \sum_{j = 1}^{m} G_j)/\Gamma = (X, \sum_{i = 1}^\nu S_i + \sum_{j = 1}^n F_j)
$$
and $(U, \Sigma_j^U)/\Gamma = (C, \Sigma_j^\calC)$ for $\Gamma$ a finite group with an essential group action on $Y \to U$. 

By Lemma 6.6 of \cite{av}, $Y$ has slc singularities. Furthermore, the sections $S_i'$ are contained in the smooth locus of the family $Y \to U$, and $G_j$ are fibers over smooth points of $U$. Therefore $(Y, \sum S_i' + \sum G_j)$ is slc. By Lemma 6.4 in \cite{av}, the pair $(X, \sum S_i + \sum F_j)$ is slc if and only if $K_X + \sum S_i + \sum F_j$ is $\Q$-Cartier. 

The property of being $\Q$-Cartier is local so we may check over a neighborhood of $p \in C$. If the point $p$ is a node or an unmarked point of $C$, then we may take a small neighborhood $V$ of $p$ avoiding the markings $\Sigma_i^C$, so that $(X_V, \sum (S_i)|_V)$ is slc by Proposition 6.10 of \cite{av}. Therefore to conclude, we only need to check in a neighborhood of a marked point $p = \Sigma_i^C$. 

Since $p$ is a smooth point of $C$, we may suppose that $C$ is smooth and that $(X \to C, \sum S_j + F)$ is a fibered surface with $F = f^{-1}(p)_{red}$ a reduced marked fiber. Furthermore, we can write $(X,\sum S_i + F) = (Y, \sum S_i' + G)/\Gamma$ as above where $\Gamma$ acts freely away from $p$ and $\Gamma_p = \Gamma$.

By Proposition 6.10 of \cite{av}, the pair $(X, \sum S_i)$ is slc and so $K_X + \sum S_i$ is $\Q$-Cartier. Thus all that remains to check is that $F$ is $\Q$-Cartier. However,  $X$ has finite quotient singularities, as it is the quotient of a family with quotient singularities by a finite group, and thus is $\Q$-factorial. Therefore, $F$ is $\Q$-Cartier. \end{proof}

\begin{cor}\label{cor:coarsespace} Let $(\calX \to \calC \to T, \calS_i, \Sigma_j^\calC)$ be a family of stack-like surfaces fibered in genus $\gamma$ curves with $\nu$ sections over a family of $n$-pointed stacky nodal curves. Then the coarse space
$$
(X \to C \to T, \sum S_i + \sum F_j)
$$
is a flat family of slc fibered surface pairs over $T$ whose construction commutes with arbitrary base change $T' \to T$. \end{cor}

\begin{proof} The stacks $\calX$ and $\calC$ are tame and so flatness of $\calX \to T$ and $\calC \to T$ implies flatness of the coarse space. Furthermore, taking coarse space commutes with arbitrary base change for tame stacks so that the fibers of the coarse space are slc pairs by the proposition. \end{proof} 

Next we need to consider positivity of $\omega_X(\sum S_i + \sum F_j)$. 

\begin{prop} \label{prop:ample} Let $(g : \calX \to \calC, \calS_i, \Sigma_j^\calC)$ be a stack-like surface fibered in genus $\gamma$ curves with $\nu$ sections over a family of $n$-pointed stacky nodal curves. Let $(f : X \to C, \sum S_i + \sum F_j)$ be the coarse fibered surface pair. Then 
$\omega_f(\sum S_i + \sum F_j)
$
is $f$-ample. 
\end{prop} 

\begin{proof} We denote $S = \sum S_i$, $F = \sum F_j$, $\calS = \sum \calS_i$, and $\Sigma = \sum \Sigma_j$. Consider the diagram 
$$
\xymatrix{\calX \ar[r]^{\psi} \ar[d]_{g} & X \ar[d]^{f} \\ \calC \ar[r]^{\pi} & C }
$$
where $\psi$ and $\pi$ are coarse moduli space maps. The morphism $\psi$ is proper and quasi-finite with branch locus contained in $F$. Letting $\widetilde{F} = \sum \widetilde{F}_j$ be marked fibers of $g$ over $F$, we have
$$
\psi^*\omega_f(S + F) = \omega_g(\calS + \widetilde{F})
$$
and it suffices to check that $\omega_g(\calS + \widetilde{F})$ is $g$-ample by base change. The fibration $g$ has reduced fibers so $\widetilde{F} = g^*\Sigma^\calC$ and for any fiber $G$ of $g$, $\omega_g(\calS + \widetilde{F})|_G = \omega_G(\calS|_G)$. It follows that $\omega_g(\calS + \widetilde{F})$ is $g$-ample since $(g : \calX \to \calC, \calS_i)$ is a family of stable pointed curves.  \end{proof} 

\begin{prop}\label{prop:stablemap}  Let $(g: \calX \to \calC, \calS_i, \Sigma_j^\calC)$ be as above and suppose the induced map $(\calC, \Sigma_j^\calC) \to \overline{\calM}_{\gamma,\nu}$ is a twisted stable map. Let $\mu: X \to \overline{M}_{\gamma,\nu}$ be the induced coarse map. Then $\omega_X(S + F)$, with notation as above, is $\mu$-ample. \end{prop}

\begin{proof} As above, we have the following diagram
$$
\xymatrix{\calX \ar[r]^{\psi} \ar[d]^{g} & X \ar[d]^{f} \\ \calC \ar[r]^{\pi} \ar[d]^{h} & C \ar[d]^{l} \\ \overline{\calM}_{\gamma,\nu} \ar[r] & \overline{M}_{\gamma,\nu}}
$$
where $l\circ f = \mu$. We need to show that $\omega_X(S + F) \otimes \mu^*H^k$ is ample for $H$ an ample line bundle on $\overline{M}_{\gamma,\nu}$ and $k$ large enough. Pulling back by the finite morphism $\psi$, it suffices to check that
$$
\psi^*(\omega_X(S + F) \otimes \mu^*H^k) = \omega_{\calX}(\calS + \widetilde{F}) \otimes \psi^*\mu^*H^k
$$
is ample as in the above proposition. Now $\psi^*\mu^*H^k = g^*\pi^*l^*H^k$. Furthermore, $\omega_\calX = g^*\omega_\calC \otimes \omega_g$ since $\calC$ is Gorenstein. 

Putting this together, we have
\begin{align*}
\omega_{\calX}(\calS + \widetilde{F}) \otimes \psi^*\mu^*H^k &= g^*\omega_\calC \otimes \omega_g(\calS + g^*\Sigma^\calC) \otimes g^*\pi^*l^*H^k \\
&= \omega_g(\calS) \otimes g^*(\omega_\calC(\Sigma^\calC) \otimes \pi^*l^*H^k) \\
&= \omega_g(\calS) \otimes g^*\pi^*(\omega_C(\Sigma^C)\otimes l^*H^k)
\end{align*}

Since $\calC \to \overline{\calM}_{\gamma,\nu}$ is a twisted stable map, we know that $\omega_C(\Sigma^C) \otimes l^*(H^k)$ is ample for $k$ large enough and so $\pi^*(\omega_C(\Sigma^C)\otimes l^*H^k)$ is ample by finiteness of $\pi$. Furthermore, $g : (\calX, \calS) \to \calC$ is a family of pointed stable curves so $\omega_g(\calS)$ is $g$-ample. Therefore,
$$
\omega_g(\calS) \otimes g^*\pi^*(\omega_C(\Sigma^C) \otimes l^*H^k)
$$
is ample for $k$ large enough. \end{proof}

This concludes the proof of Theorem \ref{thm:mainthm}. \end{proof}

\begin{cor}\label{cor:morphism} Taking the coarse space of a stack-like fibered surface induces a morphism 
$$
\varphi: \calK_{g,n}(\overline{\calM}_{\gamma, \nu}) \to \calA_{\bf{v}}(\overline{M}_{\gamma,\nu})
$$
that is locally of finite type. \end{cor}.

\begin{remark} The infinitessimal structure of the map in Corollary \ref{cor:morphism} may be complicated and we do not discuss this here. One expects this to be a finite ramified morphism, but the ramification locus may be hard to determine. See also \cite{fsv} where a related question is discussed in arbitrary dimension but without marked fibers. 
\end{remark}

\subsection{From fibered surfaces to twisted stable maps} 

Next we identify which fibered surfaces appear as above as the coarse space of a stack-like fibered surface $(\calX \to \calC \to \overline{\calM}_{\gamma,\nu}, \calS_i, \Sigma_i^\calC)$ with marked sections and marked fibers. Note that the coarse pair
$$
(f: X \to C \to \overline{M}_{\gamma,\nu}, \sum_{i = 1}^\nu S_i + \sum_{j = 1}^n F_j)
$$
consists of a fibered surface $f : X \to C \to \overline{M}_{\gamma,\nu}$ with $\nu$ sections $S_i$ and $n$ reduced marked fibers $F_j$ such that:
\begin{enumerate}
\item the generic fiber of $f$ is a stable $\nu$-pointed genus $\gamma$ curve, and 
\item every non-stable fiber of $f$ is contained in the support of the marked fibers $\sum F_j$. \end{enumerate} Furthermore, $\omega_f(\sum S_i + \sum F_j)$ is $f$-ample by Proposition \ref{prop:ample}. This motivates the following definition:

\begin{definition}\label{def:twistedsurface} Let $(f : X \to C, \sum_{i = 1}^\nu S_i + \sum_{j = 1}^n F_j)$ be an slc fibered surface pair and let $S = \sum S_i$, $F = \sum F_j$. Then $(f : X \to C, S + F)$ is said to be a \textbf{twisted fibered surface pair} if
\begin{enumerate}[label = (\alph*)]
\item the generic fiber of $f$ is a $\nu$-pointed genus $\gamma$ curve,
\item the support of every non-stable fiber is contained in the support of $F$, and
\item $\omega_f(S + F)$ is $f$-ample. 
\end{enumerate}
\end{definition} 
\begin{remark} Note that in this definition we are still requiring that our fibered surfaces $f : X \to C$ are equidimensional morphisms flat over the smooth locus of $C$. In particular, there is a unique coarse moduli space map $C \to \overline{M}_{\gamma, \nu}$ when $C$ is smooth. This data is preserved when we consider stable limits in the space of maps $\calA(\overline{M}_{\gamma,\nu})$. In Section \ref{sec:stablereduction}, we will consider stable limits of twisted surfaces without the data of the map, in which case the limits may \emph{not} be flat fibrations and so may not have a well defined coarse map to $\overline{M}_{\gamma,\nu}$. 

\end{remark} 
\begin{remark}\label{relativelystable} Condition (c) in Definition \ref{def:twistedsurface}, that $\omega_f(S + F)$ is $f$-ample, is equivalent to requiring that $\omega_X(S + F)$ be $f$-ample since $C$ is Gorenstein so $\omega_X = \omega_f \otimes f^*\omega_C$. Therefore this condition really asserts that $f : (X, S + F) \to C$ is a stable map. Consequently, we often refer to $(X, S + F)$ as the relatively stable model over $C$.  In this case, $(X, S + F)$ is uniquely determined by 
$$
X \cong \Proj_C\left(\bigoplus_{m \geq 0} f_*\calO_X(m(K_X + S + F))\right)
$$
as a fibered surface over $C$. Here $K_X$ is a $\Q$-Cartier divisor  corresponding to $\omega_X$. 
\end{remark}

The following proposition shows that Definition \ref{def:twistedsurface} completely characterizes the coarse spaces of stack-like fibered surfaces with marked fibers. 

\begin{prop}\label{prop:surj} Let $(f: X \to C, \sum_{i = 1}^\nu S_i + \sum_{j = 1}^n F_j)$ be a twisted fibered surface pair over a smooth curve $C$. Then it is the coarse space of a stack-like twisted surface $(\calX \to \calC \to \overline{\calM}_{\gamma,\nu}, \calS_i, \Sigma_j^\calC)$. 
\end{prop} 

\begin{proof} Since the generic fiber of $f$ is a $\nu$-pointed genus $\gamma$ surface, there is a rational map $C \dashrightarrow \overline{\calM}_{\gamma, \nu}$ defined on some open subset $C^0 \subset C$. By properness of the moduli stack of stable curves, there exists a twisted curve $\calC$ with coarse space $C$ and coarse map an isomorphism over $C^0$ such that $C^0 \to \overline{\calM}_{\gamma, \nu}$ extends to $\calC \to \overline{\calM}_{\gamma, \nu}$. We mark $\calC$ by the union of stacky points $\calC \setminus C_0$ and any marked points in $C_0$ lying under stable marked fibers in $\sum F_j$ and denote by $\Sigma_j^\calC$ the corresponding marked points. 

Pulling back the universal family and universal sections on $\overline{\calM}_{\gamma, \nu}$ gives us a stack-like fibered surface
$
(\calX \to \calC, \calS_i, \Sigma_j^\calC).
$
Let $(f' : X' \to C, \sum S_i' + \sum F_j')$ be the corresponding coarse fibered surface pair. Denoting $S = \sum S_i$, $F = \sum F_j$ and similarly for $X'$, there is a birational map $\mu : X' \dashrightarrow X$ over $C$ satisfying
\begin{enumerate}
\item $\mu$ is defined away from the non-stable fibers of $f$,
\item $\mu_*(S' + F') = S + F$. 
\end{enumerate}

Consider a resolution of indeterminacies of the birational morphism $\mu$.
$$
\xymatrix{ & Z \ar[rd]^\alpha \ar[ld]_\beta & \\ X' & & X}
$$
Since the pairs are slc, standard arguments show (see e.g. Lemma 6.3 of \cite{calculations}) that: 
\begin{align*}
\alpha_*\calO_Z(m(K_Z + \alpha_*^{-1} S + \alpha_*^{-1} F + \mathrm{Exc}(\alpha))) &= \calO_X(m(K_X + S + F)) \\
\beta_*\calO_Z(m(K_Z + \beta_*^{-1} S' + \beta_*^{-1} F' + \mathrm{Exc}(\beta))) &= \calO_X(m(K_{X'} + S' + F')).
\end{align*}
However, every divisor contracted by $\mu$ is marked with coefficient 1 in the boundary of $X'$ and is exceptional for $\alpha$ so it appears in $K_Z + \alpha_*^{-1} S + \alpha_*^{-1} F + \mathrm{Exc}(\alpha)$ with coefficient $1$. Similarly for divisors contracted by $\mu^{-1}$ in $X$. Therefore
$$
K_Z + \alpha_*^{-1} S + \alpha_*^{-1} F + \mathrm{Exc}(\alpha) = K_Z + \beta_*^{-1} S' + \beta_*^{-1} F' + \mathrm{Exc}(\beta).
$$
It follows that
$$
f_*\calO_X(m(K_X + S + F)) = f'_*\calO_{X'}(m(K_{X'} + S' + F')).
$$ Since the pairs are relatively stable models over $C$, they must be isomorphic with $\mu$ the isomorphism by Remark \ref{relativelystable}. \end{proof}

\section{Compact moduli of fibered surfaces}\label{sec:stablereduction}

For some purposes, it is beneficial to keep the information of the map $f : X \to C$ as part of the data parametrized by the moduli space. While the stable map $(X, S + F) \to \overline{\calM}_{\gamma, \nu}$ parametrized by $\calA_{\bf{v}}(\overline{M}_{\gamma,\nu})$ often determines the map $X \to C$, this is not always the case, as illustrated by the following example of Abramovich and Vistoli. 

\begin{example}[Abramovich-Vistoli] Let $C$ and $C'$ be two non-isomorphic stable curves of some fixed genus $g$ that become isomorphic over an algebraically closed field. Then $X = C \times C'$ has two non-isomorphic fibrations $X \to C$ and $X \to C'$ that induce the same constant map to $\overline{M}_g$. \end{example}

In this section, we study variations of the moduli problem for fibered surfaces that reintroduce the data of the fibration. 

We begin with a definition of $\calF_{g,n}^{\bf{v}}(\overline{M}_{\gamma,\nu})$. This is the moduli problem for pairs $$(f : X \to C \to \overline{M}_{\gamma,\nu}, S + F)$$ 
where $f : X \to C$ is a $\nu$-pointed genus $\gamma$ fibration with sections $S = \sum_{i = 1}^\nu S_i$ and marked fibers $F = \sum_{j = 1}^n F_j$ such that \begin{enumerate}
\item $(f : X \to C, S + F)$ is twisted, \item $C \to \overline{M}_{\gamma, \nu}$ is the coarse moduli space map, and \item $(X, S+ F) \to \overline{M}_{\gamma, \nu}$ is a stable map with volume $\bf{v}$. \end{enumerate}

\begin{theorem}\label{thm:thm2} The functor $\calF_{g,n}^{\bf{v}}(\overline{M}_{\gamma,\nu})$ is representable by a proper Deligne-Mumford stack with projective coarse space, and the morphism $\varphi: \mc{K}_{g,n}(\overline{\calM}_{\gamma,\nu}) \to \calA_{\bf{v}}(\overline{M}_{\gamma,\nu})$ factors as 
$$
\mc{K}_{g,n}(\overline{\calM}_{\gamma,\nu}) \to \calF_{g,n}^{\bf{v}}(\overline{M}_{\gamma,\nu}) \to \calA_{\bf{v}}(\overline{M}_{\gamma,\nu})
$$
where the first map is surjective and the last map is forgetting the fibration $f : X \to C$. 
\end{theorem}

\begin{proof} We can consider the product $\calA_{\bf{v}}(\overline{M}_{\gamma,\nu}) \times \calK_{g,n}(\overline{M}_{\gamma,\nu})$. Over this we have a universal surface, a universal curve, as well as universal maps to $\overline{M}_{\gamma,\nu}$. Taking the relative Hom stack over the product, we obtain a Deligne-Mumford stack $\mc{D}$ locally of finite type paramaterizing triples
$$
(X \to \overline{M}_{\gamma,\nu}, C \to \overline{M}_{\gamma,\nu}, X \to C)
$$
consisting of an Alexeev stable map of volume $v$, a Kontsevich stable map, and a morphism $X \to C$. 

Note that $\mc{F}^{\bf{v}}_{g,n}(\overline{M}_{\gamma,\nu})$ is a substack of $\mc{D}$. On the other hand, taking coarse space yields a morphism $\calK_{g,n}(\overline{\calM}_{\gamma,\nu}) \to \mc{D}$. By Proposition \ref{prop:surj}, the image of this morphism is $\mc{F}^{\bf{v}}_{g,n}(\overline{M}_{\gamma,\nu})$. Since $\calK_{g,n}(\overline{\calM}_{\gamma,\nu})$ is proper, it follows that $\mc{F}^{\bf{v}}_{g,n}(\overline{M}_{\gamma,\nu})$ is also proper. In particular, it is a closed substack of $\mc{D}$ and so is itself a Deligne-Mumford stack. 

Composing with the projection $\mc{D} \to \calA_{\bf{v}}(\overline{M}_{\gamma,\nu}) \times \calK_{g,n}(\overline{M}_{\gamma,\nu}) \to \calA_{\bf{v}}(\overline{M}_{\gamma,\nu})$ yields the map $\mc{F}^{\bf{v}}_{g,n}(\overline{M}_{\gamma,\nu}) \to \calA_{\bf{v}}(\overline{M}_{\gamma,\nu})$. The factorization is clear by construction. \\

To demonstrate projectivity, we show that the morphism $\calF_{g,n}^{\bf{v}}(\overline{M}_{\gamma,\nu}) \to \calA_{\bf{v}}(\overline{M}_{\gamma,\nu})$ is quasi-finite. This implies that it is finite on the level of coarse spaces, yielding projectivity as both spaces are proper and $\calA_{\bf{v}}(\overline{M}_{\gamma,\nu})$ has a projective coarse moduli space by \cite[Theorem 4.2]{mgnw}. This is tantamount to showing that given a twisted fibered surface pair $(f: X \to C, S + F)$, there are finitely many $\nu$-pointed genus $\gamma$-fibrations $f' : X \to C'$  over an $n$-pointed genus $g$ stable curve making $(f' : X \to C', S + F)$ into a twisted fibered surface pair. 

The key point is that the space of deformations of $(f: X \to C, S + F)$ fixing $(X, S + F)$ is zero dimensional. Given that, the statement follows as $\calF_{g,n}^{\bf{v}}(\overline{M}_{\gamma,\nu})$ is finite type. The deformations of the map $X \to C$ deform the fibers of $X \to C$ so that they remain with trivial normal bundle. Therefore, as all fibers are algebraically equivalent, they must all be contracted in any deformation. That is, any $f' : X \to C'$ deforming $f$ must factor through as $X \to C \to  C'$. This implies all deformations of $f$ fixing $(X, S + F)$ are induced by automorphisms of $C$ preserving the marked points lying under $F$. But $C$ (with these marked points) is a stable pointed curve and so it has discrete automorphism group.

\end{proof}

We note that by construction there is a morphism $\calF_{g,n}(\overline{M}_{\gamma,\nu}) \to \calK_{g,n}(\overline{M}_{\gamma,\nu})$. This map is representable by Deligne-Mumford stacks. Indeed the fiber over a stable map $g: (C,\Sigma_i^C) \to \overline{M}_{\gamma,\nu}$ is a closed substack of the stack of Alexeev stable maps to $C$ with polarization coming from $\omega_C(\Sigma^C) \otimes g^*H$ for some fixed very ample $H$ on $\overline{M}_{\gamma,\nu}$. 
\subsection{From stable maps to stable pairs} 

Another interesting variant is the moduli space of triples $(f : X \to C, S + F)$ where $(X, S + F)$ is a stable pair, that is, $(X, S + F) \to \Spec k$ is a stable map to a point. This moduli problem is obtained from the previous one by forgetting map $X \to \overline{M}_{\gamma,\nu}$ and taking the stable model (over a point) of the surface pair. One should think of this as an analogue of the morphism $\calK_{g,n}(V) \to \overline{\calM}_{g,n}$ obtained by taking a stable map $(C, \Sigma_i^C) \to V$ to the stabilization of the prestable curve curve $(C, \Sigma_i^C)$. 

However, the situation is much more complicated in dimension $2$. First, there is no reason to expect that process of stabilizing the pair $(X, S + F)$ is functorial: it may not commute with base-change in families. A more fundamental issue is that it is unclear that the stabilization of $(X, S + F)$ preserves the structure of a fibration to a curve and, even if it does, one must identify the resulting fibrations. The next proposition is a key step in resolving this issue. \\

First we recall some facts about stable curves. For a pointed nodal curve $(C, \Sigma_i^C)$, we form the dual graph of $C$ by assigning a vertex to each irreducible component of $C$ and an edge for each node. For $C_\alpha \subset C$ an irreducible component, we denote by $v_\alpha$ the valence of $C_\alpha$ in the dual graph and marked points lying on $C_\alpha$ by $n_\alpha$. Then it is clear that $(C, \Sigma_i^C)$ is stable if and only if
$
2g(C_\alpha) - 2 + v_\alpha + n_\alpha > 0
$
for all $\alpha$. Here $g(C_\alpha)$ is the geometric genus of $C_\alpha$. 

\begin{prop}\label{prop:basecurve} Let $(f : X \to C, S + F)$ be a twisted fibered surface pair. Let $C_\alpha$ be any component of $C$. Then for any component of the section $S_\alpha$ lying above $C_\alpha$, we have
$$
(K_X + S + F).S_\alpha = 2g(C_\alpha) - 2 + v_\alpha + n_\alpha.
$$
\end{prop} 

\begin{proof} First take the normalization $\nu : \bigsqcup_{\beta}C'_\beta \to C$ where $\nu_\beta : C_\beta' \to C_\beta$ is the normalization of each irreducible component of $C$ and $\nu = \bigsqcup \nu_\beta$. Consider the pullback
$$
\xymatrix{X'_\alpha \ar[r]^{\phi} \ar[d]_{f'} & X \ar[d]^f \\ C'_\alpha \ar[r]^{\nu_\alpha} & C}.
$$
Then $f'$ is a $\nu$-pointed genus $\gamma$ fibration over a smooth genus $g(C_\alpha)$ curve with sections corresponding to the components $S_\alpha$ of $S$ lying over $C_\alpha$. 

Furthermore, $\phi^*(K_X + S + F) = K_{X'_\alpha} + G + \phi^*S + \phi^*F$ where $G$ are the $v_\alpha$ many fibers lying over the points of $C'_\alpha$ mapped to the nodes of $C$. Therefore $(f' : X'_\alpha \to C'_\alpha, \phi^*S + \phi^*F + G)$ is a twisted fibered surface pair over a smooth curve with sections $\phi^*S$ and marked fibers $\phi^*F + G$. Furthermore, $\phi$ is finite with generic degree one on any section of $f'$ so
$$
(K_X + S + F).S_\alpha = (K_{X'_\alpha} + \phi^*S + \phi^*F + G).\phi^*(S_\alpha)
$$
for any component of the section $S_\alpha$ lying over $C_\alpha$.

Therefore, it suffices to prove the formula in the case of a twisted fibered surface over a smooth base curve. Thus, suppose $(f : X \to C, S + F)$ is a twisted surface over a smooth curve $C$. By Proposition \ref{prop:surj}, there is a stack-like fibered surface $(g: \calX \to \calC, \calS_i, \Sigma_j^\calC)$ whose coarse space is $(f : X \to C, S + F)$.  Consider the diagram
$$
\xymatrix{\calX \ar[r]^\psi \ar[d]_g & X \ar[d]^f \\ \calC \ar[r]^\pi & C.} 
$$
As in the proof of Proposition \ref{prop:stablemap}, the equality $\psi^*\omega_X(S + F) = \omega_\calX(\calS + \widetilde{F})$ holds, where $\widetilde{F}_1, \ldots, \widetilde{F}_n$ are the marked fibers of $g$ lying over $F_1, \ldots, F_n$. Then by the projection formula, it suffices to compute the degree of $\omega_\calX(\calS + \widetilde{F})|_{\calS_i}$ for any section $\calS_i \cong \calC$. Since $g$ is a family of stable curves, the divisor $\calS_i$ passes through the smooth locus of $\calX$, and so is a Cartier divisor. In particular, the adjunction formula holds and $\widetilde{F}_j.\calS_i = \Sigma_j^\calC$ for each $j$. Therefore, 
$$
\omega_\calX(\calS + \widetilde{F})|_{S_j} = \omega_\calC\left(\sum_j \Sigma_j^\calC\right) = \pi^*\omega_C\left(\sum_j \Sigma_j^C\right)
$$
and the desired formula follows since the degree of $\pi$ is $1$. \end{proof} 

Consequently, we see that if $(f : X \to C, S + F)$ is a twisted fibered surface, the process of taking the stable model contracts components of the section in $X$ if and only if those components of $C$ must be contracted to stabilize the curve. Therefore, the stable model of $(X, S + F)$ has a map to the stabilization of $(C, \Sigma_i^C)$. 

Let $X_\alpha \to C_\alpha$ be a component of the twisted surface. When
$
2g(C_\alpha) - 2 + v_\alpha + n_\alpha \le 0
$
so that $C_\alpha$ and all the marked sections of $X_\alpha \to C_\alpha$ get contracted, the resulting surface component must be mapped to a point in the stabilization of $C$. In particular, the stable model of $(f : X \to C, S + F)$ is no longer equidimensional morphism! This motivates the following preliminary definition of a moduli functor for twisted fibered surfaces without the map to $\overline{M}_{\gamma,\nu}$:

\begin{definition}\label{def:fibration} A \textbf{$\nu$-pointed genus $\gamma$ slc fibration} is a pair $(f : X \to C, S + F)$ where $f : X \to C$ is a projective morphism with connected fibers, $S = \sum_{i = 1}^\nu S_i$ is a sum of sections and $F = \sum_{i = 1}^n F_i$ is a sum of reduced marked vertical divisors such that
\begin{enumerate}[label = (\alph*)]
\item $(X, S + F)$ is an slc pair and $(C, \Sigma_i^C)$ is a pointed nodal curve, and
\item every component of $X$ is either a $\nu$-pointed genus $\gamma$ twisted fibered surface pair, or a surface pair contracted to a point by $f$. 
\end{enumerate}
We say $(f : X \to C, S + F)$ is \textbf{stable} if $(X, S + F)$ is a stable pair. 
\end{definition}

\begin{remark} By Proposition \ref{prop:basecurve}, the base curve $(C, \Sigma_i^C)$ of an slc fibration $(f : X \to C, S + F)$ is a stable curve. \end{remark} 

\begin{definition}\label{def:ffunctor} Consider the moduli stack of twisted stable $\nu$-pointed genus $\gamma$ slc fibrations of volume $v$ over an $n$-pointed genus $g$ stable curve as in the above definition. Denote the closure of the substack where $f$ is equidmensional and $C$ is smooth by $\calF_{g,n}^v(\gamma,\nu)$. \end{definition}

If $(f : X \to C, S + F)$ is an slc fibration where $f$ is equidimensional, then every component of $f$ is a twisted fibered surface of fixed genus and number of sections. Therefore, $(f : X \to C, S + F)$ is itself a twisted fibered surface. Thus the stack $\calF_{g,n}(\gamma,\nu)$ parametrizes stable degenerations of twisted fibered surfaces where we forget the coarse moduli map $C \to \overline{M}_{\gamma, \nu}$. As implied by Proposition \ref{prop:basecurve}, the stable limits may \emph{no longer} have the structure of a twisted fibered surface. 

\begin{prop}\label{prop:thm3} The functor $\calF_{g,n}^{v}(\gamma,\nu)$ is representable by a finite type Deligne-Mumford stack. \end{prop}

\begin{proof} The construction is analogous to that of $\calF^{\bf{v}}_{g,n}(\overline{M}_{\gamma,\nu})$. We let $\calA_v$ be the Koll\'ar-Shepherd-Barron-Alexeev (KSBA) moduli space of stable surface pairs of volume $v$, i.e. $\calA_{\bf{v}}(V)$ when $V$ is a point. Then over the product $\calA_v \times \overline{\calM}_{g,n}$ there is a universal surface, a universal curve, and the relative Hom stack $\mc{D}$ parametrizing $(f : X \to C, S + F)$ with no conditions on the map $f$. 

Now consider the locus in $\mc{D}$ where $f$ is equidimensional and $(f : X \to C, S + F)$ is a twisted fibered surface. Then there is a unique coarse moduli space map $C \to \overline{M}_{\gamma, \nu}$ which is necessarily stable since $(C, \Sigma_i^C)$ is a stable curve. Therefore this locus is the image (under the natural forgetful map) of the open substack of $\calF^{\bf{v}}_{g,n}(\overline{M}_{\gamma,\nu})$ where the base curve is stable. In particular, it is constructible in $\mc{D}$ and we take $\calF^v_{g,n}(\gamma, \nu)$ to be its closure.  \end{proof}

By construction, there are projections $\calF_{g,n}^v(\gamma,\nu) \to \calM_{g,n}$ and $\calF_{g,n}^v(\gamma, \nu) \to \calA_v$.  

\begin{remark}\label{rmk:thm3} Unlike $\calF^{\bf{v}}_{g,n}(\overline{M}_{\gamma,\nu})$, the space $\calF^v_{g,n}(\gamma,\nu)$ is not manifestly proper. In the next section we show that one can still use twisted stable maps to prove stable reduction for $\calF^v_{g,n}(\gamma,\nu)$. \end{remark}

\subsection{Stable reduction} The following proposition allows us to lift a family of twisted fibered surfaces to a family of stable maps. 

\begin{prop}\label{prop:stable1} Let $B$ be a smooth curve and let $(X \to C, S + F) \to B$ be a family of twisted fibered surfaces so that $C_\eta \to \eta$ is smooth for $\eta \in B$ the generic point. Then $(X \to C, S + F)$ is the coarse space of a family of stack-like fibered surfaces $\calX \to \calC \to B$ induced by a morphism $(\calC,\Sigma_i^\calC) \to \overline{\calM}_{\gamma,\nu}$. 
\end{prop}

\begin{proof} Let $\Sigma_i^C = f_*F_i$ be the marked points of $C$ lying under the marked fibers of $f: X \to C$. Then $X \setminus F \to C \setminus (\Sigma \cup N)$ is a family of stable curves where
$$
\Sigma := \bigcup_i \Sigma_i^C
$$
and $N$ is the necessarily finite set of points $p$ that are nodes of $C_b$ for some $b \in B$. Thus we have a map $U := C \setminus (\Sigma \cup N) \to \overline{\calM}_{\gamma,\nu}$. 

It suffices to find a stack $\calC \to B$ flat over $B$ with coarse space $\calC \to C$ an isomorphism over $U$ so that the map $U \to \overline{\calM}_{\gamma,\nu}$ extends to a representable morphism $(\calC, \Sigma_i^\calC) \to \overline{\calM}_{\gamma,\nu}$. Indeed given such an extension, there is an isomorphism
$$
\calX \cong \left(\calC \times_C X\right)^\nu
$$
where $\calX$ is the pullback of the universal family on $\overline{\calM}_{\gamma,\nu}$. Then the map $\calX \to X$ factors as $\calX \to Y \to X$ where $\calX \to Y$ is the coarse space. The map $Y \to X$ is finite and is an isomorphism over $U$. Since $X$ is normal the map $Y \to X$ is an isomorphism.

To construct the extension, consider $(C_\eta, \Sigma_i^\eta)$ the generic fiber of $C \to B$. Then there is a map $C_\eta \setminus \Sigma^\eta \to \overline{\calM}_{\gamma,\nu} \times_k \eta$. This extends uniquely to a representable morphism $(\calC_\eta, \Sigma_i^{\calC_\eta}) \to  \overline{\calM}_{\gamma,\nu} \times_k \eta$ by properness of the moduli of stable curves. Indeed $C_\eta \setminus \Sigma^\eta$ is a smooth punctured curve and so, by stable reduction, the around each puncture $\Sigma_i^C$, the morphism extends after a ramified cyclic cover of some order $n_i$. That is, the morphism extends over $\calC_\eta$ the root stack of $(C_\eta, \Sigma_i^\eta)$ of order $n_i$ at each point. 

Therefore there is an open subset $V \subset B$ such that the morphism extends to the preimage of $V$ in the corresponding root stack $\calC'$ of $(C, \Sigma_i^C)$. By assumption, the morphism was defined away from finitely many points on the complement of $V$. Since $\calC'$ is a smooth stack away from the points $N$, we may apply the purity lemma of Abramovich-Vistoli \cite[Lemma 2.4.1]{av} to extend the morphism over all points except possibly the preimages of $N$. At these points, although $\calC'$ may fail to be smooth, it will have quotient singularities. Therefore we may take the canonical stack $\calC$ of $\calC'$. Then $\calC$ is smooth so by the purity lemma the morphism extends over $\calC$, and $(\calC, \Sigma_i^{\calC}) \to B$ is a flat family of nodal stacky curves. \end{proof} 

\begin{cor} Let $(X \to C, S + F) \to B$ be a family of twisted fibered surfaces over a smooth curve $B$ such that the generic fiber of $C \to B$ is smooth. Then there exists a well defined coarse moduli space map $C \to \overline{M}_{\gamma, \nu}$. \end{cor} 

\begin{cor}\label{cor:proper} The space $\calF^v_{g,n}(\gamma, \nu)$ of slc surface fibrations is proper. \end{cor} 

\begin{proof} Let $B$ be a smooth curve with closed point $p \in B$, let $B^0 = B \setminus p$, and let
$$
(X^0 \to C^0, S^0 + F^0) \to B^0
$$
be a family of twisted surfaces with smooth generic fiber $C_\eta \to \eta$. By Proposition \ref{prop:stable1}, this family is the coarse space of a family of twisted fibered surfaces $\calX^0 \to \calC^0 \to B^0$ induced by a morphism $(\calC^0, \Sigma_i^{\calC^0}) \to \overline{\calM}_{\gamma,\nu}$. By properness of the moduli space of twisted stable maps, there is a unique extension to a family of maps
$
(\calC', \Sigma_i^{\calC'}) \to \overline{\calM}_{\gamma, \nu}
$
after a finite basechange $B' \to B$ ramified at $p$. 

Let $\calX' \to \calC'$ be the pullback of the universal family. Then the coarse space $(X' \to C', S' + F')$ is a family of slc surfaces over $B'$ with stable and normal generic fiber. Running the MMP on the family $X' \to C'$, we obtain a unique stable model.  \end{proof}

\begin{remark} Corollary \ref{cor:proper} gives us a method for computing stable limits in $\calF_{g,n}(\gamma, \nu)$ using twisted stable maps. However, the surface pairs $(X, S + F)$ associated to an slc surface fibration $(f : X \to C, S + F)$ are stable in their own right. Thus, the stable limit in $\calF_{g,n}(\gamma, \nu)$ of a family of slc fibrations is also the stable limit of the family of surfaces in $\calA_v$ once we forget the map $f$. In particular, we see that the KSBA stable limit of a family of surfaces that admit a fibered surface structure is an slc fibration, a fact that is \emph{not} obvious \emph{a priori}. 
\end{remark}
\section{Elliptic surfaces from twisted stable maps}\label{sec:elliptic}

In this section we elucidate and apply the above results in the special case of elliptic surfaces, that is, $(\gamma, \nu) = (1,1)$. This expands on the work done by La Nave \cite{ln}, who studied the moduli spaces $\calF^v_{g,0}(1,1)$ and used twisted stable maps to explicitly compute the stable limits of families of elliptic surfaces in this space. 

The case of elliptic surfaces is much easier to work with by hand for several reasons. The first is that $\overline{M}_{1,1} \cong \mb{P}^1$ so that the underlying stable map of a twisted stable map is relatively easy to understand. Furthermore, the geometry of an elliptic surface is determined by two discrete invariants that are completely classified: the possible singular fibers classified by Kodaira and Ner\'on and the degree of the \emph{fundamental line bundle} $\mb{L}$ (see Chapter III of \cite{mir3}). We review these in the next subsection. 

\subsection{Elliptic surfaces} For $f : X \to C$ a smooth, relatively minimal elliptic surface with section $S$, there are finitely many singular fibers. These consist of configurations of rational curves with dual graph given by an affine Dynkin diagram. 

Table 15.1 in \cite{silverman} gives the full classification in Kodaira notation for the fiber as well as their monodromy. Fiber types $I_n$ for $n \geq 1$ are reduced and normal crossings, fibers of type $I^*_n, II^*, III^*,$ and $IV^*$ are normal crossings but nonreduced, and fibers  of type $II, III$ and $IV$ are reduced but not normal crossings. 

The configurations of these singular fibers are intimately related to the line bundle $\mb{L}$:

\begin{definition} The \textbf{fundamental line bundle} of a twisted elliptic surface $f : X \to C$ with section $S$ is $\mb{L}:= (f'_*N_{S'/X'})^{-1}$ where $(f' : X' \to C, S')$ is the minimal semi-resolution. \end{definition} 

Here $N_{S'/X'}$ is the normal bundle of the section in $X'$ which is well defined since $X'$ is semi-smooth so $S'$ passes through the smooth locus. It turns out that $\mb{L}$ is an effective line bundle on $C$ that is independent of the choice of section \cite{mir3}. In fact, $\mb{L}$ determines the canonical bundle of an irreducible elliptic surface (see \cite[Proposition III.1.1]{mir3} and the generalization \cite[Theorem 6.1]{calculations}). For example a normal elliptic surface is rational if and only if $\deg \mb{L} = 1$. 

Returning to the question of configurations of singular fibers, the number of singular fibers in a normal elliptic surface, counted appropriately, is equal to $12\deg(\mb{L})$. Here counted appropriately means that each singular fiber is weighted by the order of vanishing of the discriminant at the corresponding point in $p$. We can read this order of vanishing from the Kodaira fiber type. 

By understanding $\deg(\mb{L})$ and the configurations of singular fibers on the elliptic surface $f : X \to C$, one obtains a method to determine the degree of the corresponding twisted stable map. For example, for a generic elliptic surface all the singular fibers are nodal elliptic curves of type $I_1$. These each contribute $1$ to $12\deg(\mb{L})$ so there are exactly $12 \deg(\mb{L})$ of them. Therefore the coarse map $C \to \overline{M}_{1,1}$ is degree $12 \deg(\mb{L})$. 

\subsection{Local analysis of twisted fibers}\label{sec:subseclocal} Next we proceed with a local analysis to illustrate explicitly how each singular fiber of a twisted surface is the coarse space of a stack-like surface. 

Suppose that $(f: X \to C , S + F)$ is the stable model of an elliptic surface with section over $C$, a DVR with marked central fiber. That is,  start with a twisted elliptic surface over the spectrum of a DVR. Theorem 1.1 of \cite{calculations} describes what the central fiber of the twisted surface $(f : X \to C, S + F)$ looks like, it has:
\begin{enumerate} \item a single irreducible component that is either a stable elliptic curve meeting the section at a smooth point, or 
\item an irreducible non-reduced curve with support $\mb{P}^1$ meeting the section at a singular point of the total space $X$. \end{enumerate}

We recall some results on the local singularities found for pairs $(X, S + F)$ in \cite{calculations}.

\begin{table}[!htb]\label{sing}
    \caption{Singularities of $II, III, IV$ and $II^*, III^*, IV^*$ fiber types}
    \begin{minipage}{.5\linewidth}
      \centering
    \begin{tabular}{|l|l|}
\hline
			&  $a = 1$				\\ \hline
$II^*$		&  $A_1$, $A_2$, $A_5$	\\ \hline	
$III^*$		&  $A_1$, $2A_3$			\\ \hline	
$IV^*$		&  $3A_2$				\\ \hline

\end{tabular}
    \end{minipage}%
    \begin{minipage}{.5\linewidth}
      \centering
\begin{tabular}{|l|l|}
\hline
			&  $a = 1$					\\ \hline
$II$		&  $A^*_1$, $A^*_2$, $A^*_5$	\\ \hline	
$III$		&  $A^*_1$, $2A^*_3$			\\ \hline	
$IV$		& $3A^*_2$					\\ \hline	

\end{tabular}
    \end{minipage} 
\end{table}

Here $A^*_{n-1}$ denotes the singularity obtained by contracting a rational $(-n)$ curve in a smooth surface. By stable reduction for families of curves, the central fiber can be filled in by a stable curve after a ramified base change of $C$.

To see this explicitly, note that the fibers have quasi-unipotent monodromy group with cyclic torsion subgroup. Then we can take a cyclic $\Z/n\Z$ cover $C' \to C$ of the base ramified at the closed point. Consider the normalization of the pullback: 
$$
\xymatrix{\widetilde{X}' \ar[r] \ar[d] & X \ar[d] \\ C' \ar[r] & C}.
$$
Here $\widetilde{X}' \to C'$ is an elliptic fibration with trivial or unipontent monodromy. Therefore the central fiber is a stable elliptic curve. We have that $(\widetilde{X}', S' + F')$ is an slc pair where $S'$ is the section, and by construction $(\widetilde{X}' \to C', S' + F')/G = (X \to C, S + F)$ where $G \cong \Z/n\Z$ is the finite part of the monodromy. 

The monodromy group determines the fiber type of $F$ and the singularities of $X$ along $F$. Indeed the Kodaira fibers are paired based on dual monodromies with $I_n$ and $I_n^*$ self dual, and $II, III, IV$ dual to $II^*, III^*, IV^*$ respectively. This explains the appearance of dual singularities $A_n$ and $A_n^*$ in the third columns of the tables for the singularities of the stable models (see Table 1). Indeed $II, III, IV$ and $II^*, III^*, IV^*$ have potentially good reduction so that $\widetilde{X}' \to C'$ above is a smooth morphism in these cases. Therefore the stable models for these fibers are the quotients of a smooth family by dual group actions and so must have dual quotient singularities. 

Now we can take the stack quotient $\calX := [\widetilde{X}'/G] \to \calC: = [C'/G]$ to obtain a stack-like fibered surface. These local models provide charts for the global surface obtained in the construction in Proposition \ref{prop:surj} of a global stack-like elliptic surface whose coarse space is a given twisted elliptic surface.

\subsection{Explicit stable reduction for elliptic surfaces} 

Finally we end with a discussion of the surfaces that appear on the boundary of $\calF_{g,n}(1,1)$, by explicitly carrying out the process described in Corollary \ref{cor:proper}. This is a direct application of the work of La Nave \cite{ln} in the case $n = 0$. In contrast to \cite{ln}, we study what happens as we keep the marked fibers with weight 1 throughout. 

\begin{definition}[See \cite{calculations}]\label{def:fibers} Let $(f : X \to C, S + F)$ be an elliptic surface over the spectrum of a DVR. We say that $f$ has 
\begin{enumerate}[label = (\alph*)]
\item a \textbf{twisted fiber} if the surface is twisted with central fiber not stable,
\item an \textbf{intermediate fiber} if the  central fiber is the blowup of a twisted or stable fiber at the point where the fiber meets the section (see Definition 4.9 of \cite{calculations}), and
\item a \textbf{Weierstrass fiber} if the central fiber is reduced and irreducible. \end{enumerate}  \end{definition}

Note that by the results of \cite{calculations}: 
\begin{enumerate} \item[(a)] the twisted fibers are irreducible non-reduced curves of arithmetic genus $1$ supported on a rational curve, 
\item[(b)] the intermediate fibers consist of the nodal union of two irreducible components $A$ and $E$, where $E$ supports either a stable genus $1$ curve or a non-reduced arithmetic genus $1$ curve, and $A$ is a reduced rational curve.
\item[(c)] Weierstrass fibers are reduced and irreducible curves of genus $1$, and therefore they are either smooth elliptic curves, nodal elliptic curves, or cuspidal cubics. \end{enumerate}

From Proposition \ref{prop:basecurve}, we know the the section of a component of a twisted slc elliptic surface may be contracted if the component is fibered over $\mb{P}^1$ with no marked fibers and  attached along one fiber. This motivates the following: 

\begin{definition}\label{def:pseudoelliptic} A \textbf{pseudoelliptic surface} is an irreducible surface $Z$ obtained by contracting the section of an elliptic surface $(f : X \to C, S)$. \end{definition} 

We are now ready to describe the surfaces that appear on the boundary of $\calF_{g,n}(1,1)$. 

\begin{definition}\label{def:brokenelliptic} A \textbf{broken elliptic surface pair} $(f: X \to C, S + F)$ is an slc elliptic fibration as in Definition \ref{def:fibration}, such that
\begin{enumerate}[label = (\alph*)]
\item $X$ consists of an slc union of elliptic surface components with all fibers as in Definition \ref{def:fibers} and pseudoelliptic surfaces associated to such elliptic surfaces which are contracted by $f$;
\item the elliptic components of $X$ are glued along twisted or stable fibers;
\item the pseudoelliptic components are of one of the following two types: 
	\begin{enumerate}[label = (\roman*)]
		\item  trees of \textbf{Type I} pseudoelliptic components constructed inductively by gluing a twisted or stable fiber of a pseudoelliptic onto the arithmetic genus $1$ component of an intermediate fiber;
		\item \textbf{Type II} pseudoelliptic components attached along precisely two irreducible (twisted or stable) fibers. 
	\end{enumerate}
\end{enumerate}

We say $(f : X \to C, S +F )$ is \textbf{stable} if $(X, S + F)$ is a stable pair. 

\end{definition}  

\begin{figure}[h!]
\centering
\includegraphics{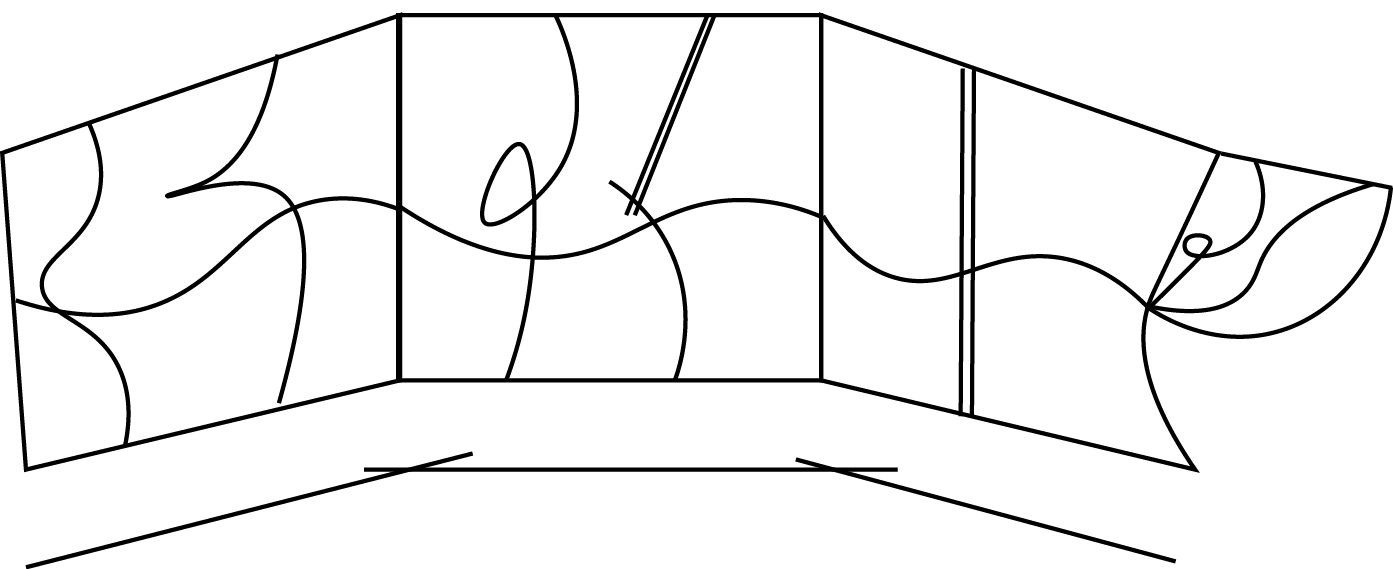}
\caption{A broken elliptic surface pair.}
\end{figure}

\begin{theorem}\label{thm:elliptic} Let $(X^0 \to C^0, S^0 + F^0) \to B^0$ be a flat family of twisted stable elliptic surface pairs over a smooth punctured curve $B^0 = B \setminus p$. Then after a finite base change, the central fiber can be filled in uniquely by a stable broken elliptic surface pair. 

\end{theorem} 

\begin{proof} Proceeding as in the proof of Corollary \ref{cor:proper}, we lift this family of twisted surfaces to a family of stack-like fibered surfaces $\calX^0 \to \calC^0 \to B^0$. Replacing $B$ with a finite base change, we have a unique way to extend the family of twisted stable maps to $(\calC, \Sigma_i^\calC) \to \overline{\calM}_{\gamma, \nu}$. The central fiber of the family of coarse surfaces is the unique limit in the space $\calF_{g,n}(\overline{M}_{\gamma,\nu})$ of Alexeev stable maps. 

Denote this family of coarse surfaces by $(X \to C, S + F) \to B$. Then the central fiber, denoted by $(f_0 : X_0 \to C_0, S_0 + F_0)$ is a twisted elliptic surface. In particular, it is equidimensional, relatively stable, with each fiber either stable or twisted. Then running the MMP on the total space of this family, we must contract any components of $C_0$ that are not stable, as well as the components of the section lying above them.

Suppose $Z$ is such a component of the central fiber attached to the rest of the central fiber $Y$ along a single twisted or stable fiber $G$. By \cite[Theorem 7.1.2]{ln}, the contraction of the section of $Z$ is a log flipping contraction of the total space $X$, whose flip is the blowup on $Y$ of the fiber $G$ to an intermediate fiber. Thus any trees of components $Z$ fibered over rational curves with no marked fibers result in trees of Type I pseudoelliptic components attached along intermediate fibers. 

After this sequence of flips, we are left with the prestable components of $C_0$, that is, the $(K_X + S + F)$-trivial components of the section. Such a component of the section lies on an elliptic components $Z$ of the central fiber over a rational component of $C_0$ which is attached along two twisted fibers. The contraction of such a section component is a log canonical contraction leading to a Type II pseudoelliptic attached along two twisted fibers. 

Finally, by \cite[Proposition 7.4 and Section 8.3]{calculations}, a rational pseudoelliptic component may be contracted by the log canonical linear series either onto the intermediate fiber it is attached along, or onto a point. The latter case results in the contraction of the component $E$ of the intermediate fiber leading to a Weierstrass fiber. \end{proof}

\bibliographystyle{alpha}
\bibliography{master}

\end{document}